\documentclass[12pt]{amsart}

\newcommand{\C}{\mathbb{C}}
\newcommand{\R}{\mathbb{R}}

\newcommand{\SL}{\operatorname{SL}}
\newcommand{\SU}{\operatorname{SU}}

\newcommand{\id}{I}

   \newtheorem{theorem}{Theorem}[section]

   \newtheorem{lemma}[theorem]{Lemma}
   \newtheorem{defn}[theorem]{Definition}

   \newtheorem{remark}[theorem]{Remark}

\begin{document}

\date{\today}

\title[Smyth-type potentials]{Triviality of the dressing isotropy 
for a Smyth-type potential and nonclosing of the resulting CMC surfaces}

\author{Josef F. Dorfmeister}
\address{}
\email{}


\author{Wayne Rossman}
\address{}
\email{}

\maketitle

\section{Introduction}

In \cite{51} we have investigated spacelike surfaces in Minkowski 3-space
which are generated by the potential 
\[ \xi = \lambda^{-1} \begin{pmatrix} 0 & 1 \\ 
 z^{-1} & 0 \end{pmatrix} dz \; . \]
This was motivated by the quantum cohomology of $\C P^1$.
It turns out that this potential yields, 
via the loop group method \cite{49}, spacelike CMC surfaces 
in Minkowski 3-space, for which the metric is invariant under 
a 1-parameter family of isometries of the domain. In 
$\R^3$, Delaunay surfaces and 
Smyth surfaces are known to have this property. Moreover, 
these surfaces can be  defined on a (punctured) disk, i.e. 
the immersion closes around the fixed point of the rotations under 
which the metric is invariant (and, in the Smyth case, is even defined 
at the fixed point). We were therefore interested in 
finding out whether among the immersions generated 
by the potential $\xi$ above 
(which is different from the Delaunay potentials and the Smyth potentials)
there are immersions closing around $z=0$. In this note we prove, 
as already partially announced in \cite{51}, that in 
every integrable surface class for which the 
potential $\xi$ above makes sense\footnote{Following 
Kobayashi \cite{**}, we are only interested in the almost compact cases, 
and hence we are interested in the involutions C1, C2, C3 and C4 in 
that paper.  However, we omit the case C4, since there we have a situation
completely different from the other three.  The groups under consideration 
for the cases C1, C2, C3 are 
$\SU(2)$, $\SU(1,1)$, and $\SL_*(2,\R)$ (isomorphic to $\SL_2 \R$ 
under conjugation by the diagonal matrix with diagonal 
entries $\sqrt{i}^{-1}$ 
and $\sqrt{i}$), respectively.   The Gauss maps, respectively, 
go into the symmetric spaces $S^2$ equal to 
$\SU(2)$ modulo diagonal matrices, 
$H^2$ equal to $\SU(1,1)$ modulo diagonal matrices, 
and $S^{1,1}$ equal to $\SL_*(2,R)$ modulo diagonal matrices.  
Therefore, harmonic maps into these 
spaces all have normalized potentials which are arbitrary meromorphic and 
off-diagonal.}, there do not exist any closing immersions 
(Theorem \ref{2}).
The conclusion comes quite easily from the fact that the
isotropy group of the dressing action is trivial (Theorem \ref{firsttheorem}).
The proof of the latter statement covers most of this note.

The authors are grateful to Martin Guest for his support of this research, 
including many detailed discussions and insightful comments.  The authors are 
also thankful for the grant from the Japan Society for the Promotion of Science 
that brought the first author to Japan.  

\section{The DPW method}\label{sect2}

Let $G$ denote any of the real Lie groups $\SU(2)$, $\SU(1,1)$ or 
$\SL_2 \R$.
Clearly, G is a real form of $Sl_2\C$. Let
$\varLambda \SL_2\mathbb{C}_{\sigma}$ denote the 
twisted loop group of smooth maps 
from the unit circle $\mathbb{S}^1$ to $\SL_2\mathbb{C}$, where "twisted" 
means that the two diagonal elements of the image of the map are even 
functions of $\lambda \in \mathbb{S}^1$ and the 
two off-diagonal elements are odd functions of $\lambda$.  
Let $\varLambda G_{\sigma}$ denote the subgroup of maps 
from $\mathbb{S}^1$ to $G $.  
Let $\varLambda_+\SL_2\mathbb{C}_{\sigma}$, resp. 
$\varLambda_+^\mathbb{R}\SL_2\mathbb{C}_{\sigma}$, denote the subgroup of 
unnormalized, resp. normalized, positive loop groups -- in other words, 
$B \in \varLambda_+\SL_2\mathbb{C}_{\sigma}$, resp. 
$B \in \varLambda_+^\mathbb{R}\SL_2\mathbb{C}_{\sigma}$, if $B$ can be 
extended smoothly to a map defined on the unit disk (with boundary 
$\mathbb{S}^1$) and 
the diagonal matrix $B|_{\lambda=0}$ does not necessarily, resp. does 
necessarily, have positive reals on the diagonal.  

At this point, the DPW method requires an Iwasawa 
splitting relative to $\varLambda G_{\sigma}$.
It turns out that in the case $G = \SU(2)$ every element $g(\lambda)$ in 
$\varLambda \SL_2\mathbb{C}_{\sigma}$ can be 
written in the form $ g = F B$ with $F \in 
\varLambda SU(2)_{\sigma}$ and $B \in \varLambda_+\SL_2\mathbb{C}_{\sigma}$. 
In the other two cases, however, this is not true. More precisely, in the
case of $G=\SL_2 \mathbb{R}$ there is a cell that is open and dense in 
$\varLambda \SL_2\mathbb{C}_{\sigma}$ and 
on which the Iwasawa splittings exist, but 
which cannot be all of $\varLambda \SL_2\mathbb{C}_{\sigma}$.  In the 
case of $G=\SU(1,1)$ there are two open cells, 
$\mathcal{B}_1 = \varLambda G_{\sigma} \cdot  
\varLambda_+\SL_2\mathbb{C}_{\sigma}$
and $ \mathcal{B}_2 = 
\varLambda G_{\sigma} \cdot \omega \cdot 
\varLambda_+\SL_2\mathbb{C}_{\sigma}$ with
$\omega = \left( \begin{array}{cc}
0 & - \lambda \\
\lambda^{-1} & 0 \\
\end{array}
\right)$, and the union of $\mathcal{B}_1$ and $\mathcal{B}_2$
is open and dense in $\varLambda \SL_2\mathbb{C}_{\sigma}$, so 
that Iwasawa splittings 
exist for any element of $\mathcal{B}_1 \cup \mathcal{B}_2$ 
\cite{45}, \cite{48}.
With this notation the Iwasawa splitting of an element 
$L \in \mathcal{B}_1 \subset \varLambda \SL_2\mathbb{C}_{\sigma}$, 
respectively 
$L \in \mathcal{B}_2 \subset \varLambda \SL_2\mathbb{C}_{\sigma}$, is 
\[ L = F \cdot B \; , \;\;\; \text{for some} \;\; 
F \in { \varLambda \SU(1,1) }_{\sigma} \;\; \text{and} \;\; 
B \in \varLambda_+^\mathbb{R}\SL_2\mathbb{C}_{\sigma} \; , \]
respectively, 
\[ L = F \cdot \omega \cdot  B \; , \;\;\; \text{for some} \;\; 
F \in {\varLambda \SU(1,1) }_{\sigma}\;\; \text{and} \;\; 
B \in \varLambda_+^\mathbb{R}\SL_2\mathbb{C}_{\sigma} \; . \]  

Now, the DPW method can start with an equation of the form 
\[ dL = L \cdot \xi \; , \;\;\; \xi = \lambda^{-1}\begin{pmatrix}
       0&g\\ h&0    \end{pmatrix} dz
\] defined on some domain of $\mathbb{C}$ over which $g$ and $h$ 
are holomorphic and $g$ is nonzero.  
One then Iwasawa splits a solution $L$ into $L = F B$, or $F\omega B$ 
and inserts $F$ or $\omega^{-1} F \omega$ into the Sym-Bobenko formula.
In the cases $G = \SU(2)$ and $G = \SU(1,1)$ this formula is

\[ f = \frac{-i}{2H} \left[ F  
\begin{pmatrix}       1&0\\ 0&-1    \end{pmatrix} 
F^{-1} - 2 \lambda (\partial_\lambda F) F^{-1} 
\right] \] 
to obtain (for every $\lambda \in S^1$) a conformal spacelike CMC 
$H \neq 0$ immersion $f$ into Euclidean $3$-space $\mathbb{R}^3$ and 
Minkowski $3$-space $\mathbb{L}^3$ respectively.  

In the case $G = \SL_2 \R$, the Sym formula
\[ f = -i \left[ \lambda (\partial_\lambda F) F^{-1} \right]  \] 
yields a conformal timelike
constant negative Gau{\ss} curvature 
$K < 0$ immersion $f$ into $\mathbb{L}^3$.  See \cite{**}.  

To obtain the Smyth surfaces in $\R^3$, we can 
take the domain to be all of $\mathbb{C}$ and the potential to be 
\[ \xi_k = \lambda^{-1}\begin{pmatrix}
       0&1\\ c z^k & 0    \end{pmatrix} dz
     \; , \;\;\; k \in \mathbb{Z} \; , \;\;\; k \geq 0 \; , \]  
where $c \in \mathbb{C} \setminus \{ 0 \}$, and we can take 
the initial condition for the solution $L$ to be $L|_{z=0}=\id$. 

The case we will consider in this note is  

\[ \xi = \lambda^{-1}\begin{pmatrix}
       0&1\\ c z^{-1} & 0    \end{pmatrix} dz
     \;  \;\;\;  \;  \;\;\; \;  \]  
where $c \in \mathbb{C} \setminus \{ 0 \}$,
and now we cannot specify any initial condition at $z=0$.

\section{The result on dressing isotropy}
\label{section2}

As pointed out in the introduction, we consider all integrable surface classes 
for which the potential $\xi$ makes sense. In the classification of \cite{**}
these are the CMC surfaces in $\R^3$, spacelike CMC surfaces in 
Minkowsi 3-space
$\mathbb{L}^3$ and the timelike surfaces of constant negative 
Gau{\ss} curvature 
in $\mathbb{L}^3$.  The group $G$ in each case is $\SU(2)$, 
$\SU(1,1)$ and $\SL_2 \R$, respectively. However, for the claim 
and the proof of Theorem \ref{firsttheorem} this is of no importance. 
It is only important to note that our potential is 
contained in the complexification 
$p^{\C}$, where $g = k + p$ ($g$ denotes the Lie algebra in any of the three 
cases, and $k$ is the diagonal part and $p$ is the off-diagonal 
part) is the Cartan decomposition 
corresponding to the target space of the Gau{\ss} map.
In all three cases $p^{\C}$ consists of all $2 \times 2$ off-diagonal 
matrices with complex entries.

\begin{theorem}\label{firsttheorem}
Let $L$ be any solution of $dL=L \xi $ for 
the potential \[ \xi = \lambda^{-1}\begin{pmatrix}
 0&1\\ c z^{-1} &0    \end{pmatrix} dz,
\] where $c \in \C \setminus \{ 0 \}$.  Then the isotropy group of $L$ 
relative to the dressing action is $\{ \pm \id \}$.  
\end{theorem}

We now consider how to prove this theorem.  
Solving $dL=L \xi$, where $\xi$ is as in the theorem 
and $c \in \mathbb{C} \setminus \{ 0 \}$,
we have solutions of the form 
\[
L=\begin{pmatrix}
     X^{\prime}&\lambda^{-1}X\\
     \lambda Y^{\prime}&Y
     \end{pmatrix} \; , \]
and $X$ and $Y$ satisfy
\begin{equation}\label{eq:xy}
z X^{\prime\prime}-\lambda^{-2} c X=0 \; , \;\;\;  
z Y^{\prime\prime}-\lambda^{-2} c Y=0 \; . 
\end{equation}
The Frobenius method leads us to one particular solution 
\begin{equation}\label{smythp}
 \tilde L = \hat L \cdot P\; ,
\end{equation}
where
\[ \hat L=\begin{pmatrix}
       1&0\\
       \lambda^{-1}c \log z&1
       \end{pmatrix}=e^{\log z\cdot D} \; , \;\;\; 
   D=\begin{pmatrix}
          0&0\\
          \lambda^{-1}c&0 
         \end{pmatrix} \; , \]
and, for appropriate constants $\eta_{ij}$,  
\begin{equation*}
P=\begin{pmatrix}
           1&0\\
           -\lambda\eta_{2,1}-\lambda^{-1}c&1
           \end{pmatrix}
\cdot 
\begin{pmatrix}
     \sum_{j=0}^{\infty}\left( j+1\right)\eta_{1,j}z^j&\lambda^{-1} z 
     \sum_{j=0}^{\infty}\eta_{1,j}z^j\\
    \lambda\left\{\sum_{j=1}^{\infty}j\eta_{2,j}z^{j-1}+
    \lambda^{-2}c\sum_{j=0}^{\infty}\eta_{1,j}z^j\right\}&
    \sum_{j=0}^{\infty}\eta_{2,j}z^j                                  
             \end{pmatrix}\;.
\end{equation*}

Note that in this solution we can assume $\eta_{1,0} = \eta_{2,0} = 1$, 
i.e. that $\lim_{z\rightarrow 0} P = \id$.  

Let us consider the isotropy group of $\tilde{L}$: 

\begin{defn}\label{isotropy} An element 
$h \in \Lambda \SL_2\mathbb{C}_\sigma$ is in the {\em isotropy group} 
of $\tilde L$ if there exists a possibly $z$-dependent function $W_+ \in 
\Lambda_+ \SL_2\mathbb{C}_\sigma$ so that 
\[ h\tilde{L}=\tilde{L} W_+ \; . \]
\end{defn}

\begin{lemma}\label{isotropylemma}
If $h$ is in the isotropy group of $\tilde L$, then 
\begin{equation}\label{eq:hsl2}
h \in \varLambda_+\SL_2\mathbb{C}_{\sigma}\;,
\end{equation} 
and $\lim_{\lambda\rightarrow 0}h= \pm \id$, 
which implies either 
$h$ or $-h$ lies in $\varLambda_+^\mathbb{R}\SL_2\mathbb{C}_{\sigma}$.  
\end{lemma}

\begin{proof}
If $h = (h_{ij})_{i,j=1}^2$ is in the isotropy group, then 
$\tilde L^{-1} h \tilde L = 
P^{-1}e^{-\log z\cdot D}he^{\log z\cdot D} P \in\varLambda_+
\SL_2\mathbb{C}_{\sigma}$.  
We have
 \[ \tilde L^{-1} h \tilde L = \begin{pmatrix}
  h_{11}+\frac{ch_{12}}{\lambda}\log z&h_{12}\\
  h_{21}-\frac{c\left(h_{11}-h_{22}\right)}{\lambda}
  \log z -\frac{c^2h_{12}^2}{\lambda^2}\left( \log z\right)^2&h_{22}-
  \frac{ch_{12}}{\lambda}\log z
 \end{pmatrix} + \mathcal{O} \; , \] 
where $\mathcal O$ denotes 
terms converging to $0$ as $z \to 0$.  
This matrix lies in $\varLambda_+\SL_2\mathbb{C}_{\sigma}$. 
Considering the $(1,2)$-entry we observe that $h_{12} \in \mathcal{A}_+$, the 
algebra of positive Wiener functions.  Hence we have the expansion 
$h_{12}=h_{12,1}\lambda^{1}+h_{12,3}\lambda^{3}+h_{12,5}\lambda^5+\cdots \;$
for certain constants $h_{12,k}$.

Substituting this into the diagonal terms of the matrix above we see that 
$h_{11}$ and $h_{22}$ do not contain any negative powers of $\lambda$ 
in their Fourier expansions. Moreover, the terms independent of $\lambda$ are 
$h_{11,0} + c h_{12,1} \text{log}(z) + \mathcal{O}$ and 
$h_{22,0} - c h_{12,1} \text{log}(z) + \mathcal{O}$. Since the 
matrix under consideration has
determinant equal to 
$1$ and is in $\varLambda_+\SL_2\mathbb{C}_{\sigma}$, we obtain 
$(h_{11,0} + c h_{12,1} 
\text{log}(z) + \mathcal{O})(h_{22,0} - c h_{12,1} \text{log}(z) + 
\mathcal{O}) = 1$.
Hence $h_{12,1}=0$.

Substituting this  into the $(2,1)$-entry of the matrix 
above we infer that the 
third term is of order 4 in $\lambda$, while the term
$(c (h_{11,0} - h_{22,0}) / \lambda ) \text{log}(z)$ 
cannot be cancelled by $h_{21,-1}$, since the latter 
does not depend on z. Hence $h_{11,0} = h_{22,0}$
and $h_{21}$ does not contain any negative powers of $\lambda$. 
In particular, we 
have shown that h is in $\varLambda_+\SL_2\mathbb{C}_{\sigma}$. Moreover,
using the last equality and the determinant being $1$, we obtain that the 
$\lambda$-independent summand of $h$ is $\pm I$. This proves the claim.
\end{proof}

The next result closely follows arguments from \cite{47}: 

\begin{lemma}\label{lem3pt4}
Suppose $h$ is in the isotropy group of $\tilde L$, and so there 
exists a $W_+$ as in Definition \ref{isotropy}.  
If the upper-right entry of 
$W_+$ is identically zero, then $h = \id$ or $h = -\id$.  
\end{lemma}

\begin{proof} 
We note that the constant $c$ in $\xi$ can be removed by some constant gauge
and a coordinate transformation.  We therefore assume that $c = 1$.  
We have $h \tilde L = \tilde L W_+$, and we write the components of 
$W_+$ as (now the notation "$c$" plays a different role) 
\[ W_+ = \begin{pmatrix}
a & b \\ c & d 
\end{pmatrix} \; . \]
Then 
\[ 
W_+ \cdot \lambda^{-1} \begin{pmatrix}
0 & 1 \\ \frac{1}{z} & 0 
\end{pmatrix} = 
W_+ \cdot \tilde L^{-1} \partial_z \tilde L = 
W_+ \cdot (h \tilde L)^{-1} \partial_z (h \tilde L) = \]\[ 
W_+ \cdot (\tilde{L} W_+)^{-1} \partial_z (\tilde L W_+) = 
\lambda^{-1} \begin{pmatrix}
0 & 1 \\ \frac{1}{z} & 0 \end{pmatrix} W_+ +  \partial_z W_+ \; . \]  
Thus 
\begin{equation}\label{eqn:app4proof-1}
\lambda \partial_z a = - \lambda \partial_z d = \frac{b}{z} - c 
\end{equation}
and 
\begin{equation}\label{eqn:app4proof-2}
\lambda \partial_z b = - \lambda z \partial_z c = a - d 
\end{equation}
hold.  Now if $b=0$, then Equations \eqref{eqn:app4proof-1}
and \eqref{eqn:app4proof-2} give that $a=d=\pm 1$ and $c=0$.  
Therefore, $W_+=\pm \id$ and so $h = \pm \id$.  
\end{proof}

We now give a proof of Theorem \ref{firsttheorem}.  

\begin{proof}
It suffices to prove the result for one particular solution, 
as different solutions give conjugate isotropy groups, 
so let us take the solution $\tilde L$ as given above.  
We give a proof by contradiction.  Suppose $h$ is in the isotropy 
group, but not in $\{ \pm \id \}$.  The previous 
two lemmas 
imply we may assume that $h \in \varLambda_+\SL_2\mathbb{C}_{\sigma}$ 
and that $b$  is not identically zero for the matrix 
\[ W_+ = \begin{pmatrix}
a & b \\ c & d 
\end{pmatrix} \]
corresponding to $h$.  
Equations \eqref{eqn:app4proof-1} and \eqref{eqn:app4proof-2}
imply that 
\begin{equation}\label{eqn:app4proof-3} 
\frac{\lambda^2}{2} \partial_z^3 b + 
\frac{1}{z^2} b- \frac{2}{z} \partial_z b = 0 \end{equation}
holds.  Now we consider the expansion 
$b = \sum_{n=0}^\infty b_n(z) \lambda^n$ and choose $N 
\in \mathbb{Z}_+ \cup \{ 0 \}$ so that 
$b_n = 0$ for all $n<N$ and $b_N \neq 0$.  Then Equation 
\eqref{eqn:app4proof-3} implies 
\[ 2 z \partial_z b_N = b_N \] and thus 
$b_N = c_1 \sqrt{z}$ for some constant $c_1$.  
However, $W_+=\tilde{L}^{-1} h \tilde L$, and 
$\tilde{L}^{-1} h \tilde L$ is comprized of only products and 
sums of holomorphic functions and $\log z$ 
and $(\log z)^2$, so in particular $b_N = c_1 \sqrt{z} = 
f_1(z) + f_2(z) \log z + f_3(z) (\log z)^2$ for functions $f_j$ 
that are holomorphic at $z=0$.  That this is a contradiction can 
be seen as follows:  Letting $\tau$ be the deck transformation 
associated to a counterclockwise loop about the origin in 
$\C \setminus \{ 0 \}$, and applying $\tau$ twice, 
$\sqrt{z}$ is invariant, and so 
\[ f_1 + f_2 \cdot \log z + f_3 \cdot (\log z)^2 = 
   f_1 + f_2 \cdot (\log z+4 \pi i) + f_3 \cdot (\log z+4 \pi i)^2 
\; . \]  Thus 
\[ f_2 + 4 \pi i f_3 + 2 f_3 \cdot \log z = 0 \; . \]  
However, $\log z$ is not holomorphic at $z=0$, so $f_3$ must 
be identically zero, and then $f_2$ must be zero as well.  Thus 
$\sqrt{z} = f_1$, but $\sqrt{z}$ is also not holomorphic at $z=0$, 
providing the contradiction.  
\end{proof}

\section{Nonclosing of the resulting surfaces}

In this section we prove a result about the corresponding surfaces.
We recall that the potential $\xi$ considered throughout
produces surfaces of three different types via the DPW method.

\begin{theorem}\label{2}
For the potential \[ \xi = \lambda^{-1}\begin{pmatrix}
  0&1\\ c z^{-1} &0    \end{pmatrix} dz,
\] where $c \in \C \setminus{0}$, no 
resulting immersion of any of the three integrable surface types 
can be well-defined on any annular domain $\Sigma$ 
in $\mathbb{C} \setminus \{ 0 \}$ with nontrivial winding order 
about $z=0$, for any value of the associated spectral parameter 
in $\mathbb{S}^1 = \{ \lambda \in \mathbb{C} \, | 
\, |\lambda| = 1 \}$.
\end{theorem}

To prove this theorem, we again assume without loss of generality 
that $c=1$ (as in the proof of Lemma \ref{lem3pt4}), and we will 
suppose there exists a solution $L = C\tilde{L}$ defined on 
$\tilde \Sigma$ (the universal cover of $\Sigma$), for some 
$C\in\varLambda\SL_2\mathbb{C}_{\sigma}$, of $dL=L\xi$ with 
Iwasawa splitting (with respect to $G$) 
\[ L=F\cdot B \] 
so that the frame $F$ produces, at some $\lambda=\lambda_0 
\in \mathbb{S}^1$, a well-defined CMC or CGC immersion on an annular 
region $\Sigma$ with nontrivial winding number about $z=0$, and 
then find a contradiction.  

\begin{remark}
If the ambient space is $\mathbb{R}^3$, then the 
Iwasawa splitting for $\SU(2)$ is global \cite{49}, and so the 
domain of definition of the surface can be extended from 
$\Sigma$ to all of $\mathbb{C} \setminus \{ 0 \}$.  
In the other two cases (surfaces in $\mathbb{L}^{3}$), however, the 
Iwasawa splitting is not global and one can expect to encounter 
singularities on the surface, equivalently one can expect to leave the region 
where Iwasawa splittings exist, as one extends $\Sigma$ to 
larger domains within $\mathbb{C} \setminus \{ 0 \}$ (see 
\cite{48}).  For this 
reason the restriction to annular regions $\Sigma$ smaller than 
$\mathbb{C} \setminus \{ 0 \}$ is necessary in Theorem \ref{2}.  
\end{remark}

\begin{remark}
In the case of $G=\SU(1,1)$, the splitting can take two 
possible forms: either $L=FB$ or $L = F \omega B$, where 
$\omega$ is as in Section \ref{sect2}, see \cite{48}, \cite{51}.  
However, if we find that $L$ satisfies the second form, 
we can replace $L$ with $\omega^{-1} L$ (this changes the 
resulting surface only by 
a rigid motion) and $\omega^{-1} F \omega$ with $F$ to switch 
over to the first form.  So without loss of generality we may 
assume the first form.  
\end{remark}

Now, $C$ can be Iwasawa decomposed into parts 
$C_u\in\varLambda G_{\sigma}$ and 
$C_+\in\varLambda_+^{\mathbb{R}}\SL_2\mathbb{C}_{\sigma}$, i.e.
$$C=C_u\cdot C_+ \;\;\; \text{or} \;\;\; C=C_u \omega \cdot C_+ \;.$$ 
The $C_u$ or $C_u \omega$ part only moves the resulting surface in 
$\mathbb{R}^3$ or $\mathbb{L}^{3}$ by a rigid motion, so we can take 
$C$ to be in $\varLambda_+^{\mathbb{R}}\SL_2\mathbb{C}_{\sigma}$ without 
loss of generality (the possibility of $C_u \omega$ occurs only in the 
$G=\SU(1,1)$ case.).  

Let $\tau$ be the deck transformation 
associated to a counterclockwise loop about the origin in 
$\Sigma$.  Let $\mathcal{M}_{L}$, 
$\mathcal{M}_{\tilde L}$ ($\tilde L$ as in Equation 
\eqref{smythp}) and $\mathcal{M}_F$ be the monodromies 
of $L$, $\tilde L$ and $F$, respectively, with respect to the 
deck transformation $\tau$. 
That is to say, under the deck transformation $\tau$, we have the 
following transformations: 
\[ L \to \tau^* L = \mathcal{M}_{L} \cdot L \; , \;\;\; \tilde{L} \to 
\tau^* \tilde{L} = \mathcal{M}_{\tilde{L}} \cdot \tilde{L} \; , 
\;\;\; F \to \tau^* F = \mathcal{M}_{F} \cdot F \; . \]
Because $\tau^* P = P$, we have 
\begin{equation}\label{eq:mono}
\mathcal{M}_{\tilde L}
=\hat L\left(\tau(1),\lambda\right) 
\cdot \left(\hat L(1,\lambda)\right)^{-1}\\
=\hat L\left(\tau(1),\lambda\right)\\
=\begin{pmatrix}
                     1&0\\
                     2\pi i \lambda^{-1}&1
                    \end{pmatrix} \; . 
\end{equation}
The monodromy $\mathcal{M}_{L}$ is conjugate to 
$\mathcal{M}_{\tilde{L}}$ under conjugation by $C$, i.e.
\[ \mathcal{M}_{L}=C\cdot \mathcal{M_{\tilde{L}}}\cdot C^{-1} \; . \]

\begin{remark}
Note that $\mathcal{M}_F$ is a monodromy 
because the resulting immersion is assumed to be 
well-defined on $\mathbb{C} \setminus \{ 0 \}$, that is, 
$\mathcal{M}_F$ does not depend on $z$.  In other words, we have the 
following statement:  Given a CMC surface well-defined on $\Sigma$, 
the Maurer-Cartan form $F^{-1} dF$ of its extended frame $F$ 
is invariant under $\tau$. 
\end{remark}

\begin{lemma}
$\mathcal{M}_F = \pm \mathcal{M}_L$.  In particular, 
$\mathcal{M}_L$ lies in ${ \varLambda G }_{\sigma}$.  
Moreover, $\tau^* B = \pm B$.  
\end{lemma}

\begin{proof}
$\tau^* L = \mathcal{M}_L \cdot L = 
\tau^* F \cdot \tau^* B = \mathcal{M}_F \cdot L B^{-1} \cdot \tau^* B$ 
implies $\mathcal{M}_F^{-1} \mathcal{M}_L \cdot 
L = L W$, where $W = B^{-1} \cdot \tau^* B$ is a positive loop.  
By Theorem \ref{firsttheorem}, the isotropy group is 
$\{ \pm \id \}$, implying the lemma.  
\end{proof}

Note that 
\[ \tau^* f = \mathcal{M}_L f \mathcal{M}_L^{-1} + H^{-1} \partial_t 
\mathcal{M}_L \cdot \mathcal{M}_L^{-1} \] 
for $G = \SU(2)$ or $G=\SU(1,1)$, and for 
$\lambda =e^{i t}$.  The term $\mathcal{M}_L f \mathcal{M}_L^{-1}$ 
represents a rotation of $f$, and the term $\partial_t 
\mathcal{M}_L \cdot \mathcal{M}_L^{-1}$ represents a translation.  
Thus if $\tau^* f = f$ for $\lambda = \lambda_0=e^{i t_0}$, we have 
\begin{equation}\label{trinoidstyleclosingconditions}
\mathcal{M}_L|_{\lambda=\lambda_0} = \pm \id \; , \;\;\; 
\partial_t \mathcal{M}_L|_{\lambda=\lambda_0} = 0 \; . 
\end{equation}

When $G=\SL_2 \R$, we have 
\[ \tau^* f = \mathcal{M}_L f \mathcal{M}_L^{-1} 
- \partial_t \mathcal{M}_L \cdot \mathcal{M}_L^{-1} \; . \] 

We finally prove Theorem \ref{2}.  

\begin{proof}
By way of contradiction, assume there exists a solution $L$ that 
gives such a surface for the case of 
$G=\SU(2)$ or $G=\SU(1,1)$.  Then $\mathcal{M}_{\tilde L}$ as in 
\eqref{eq:mono} implies $\mathcal{M}_L = C \cdot 
\mathcal{M}_{\tilde L} \cdot C^{-1}$ is never $\pm \id$ at any 
$\lambda_0$.  But if $f|_{\lambda_0}$ closes to be well-defined on 
an annulus, then Equation \eqref{trinoidstyleclosingconditions} 
implies that $\mathcal{M}_L$ must be $\pm \id$ at $\lambda_0$.  This 
contradiction proves the theorem for these two cases.  

In the third case $G=\SL_2 \R$, $\mathcal{M}_{\tilde L}$ as in 
\eqref{eq:mono} is conjugated by an element $C \in 
\varLambda_+^{\mathbb{R}}\SL_2\mathbb{C}_{\sigma}$ to 
$\mathcal{M}_{F} \in \varLambda G_{\sigma}$, which is a contradiction.  
\end{proof}

\section{A remark on the more general potentials $\xi_k$}

Finally, in this section we remark on the behavior that results with 
the more general potential 
\[ \xi_k = \lambda^{-1} \begin{pmatrix} 0 & 1 \\ 
 c z^{k} & 0 \end{pmatrix} dz \; , \;\;\; k \in \mathbb{Z} \; , \;\;\; 
c \in \mathbb{C} \setminus \{ 0 \} \; . \]  
We will see that any value of $k$ other than $-1$ can produce CMC 
surfaces in $\R^3$ that close on annular domains  
in $\mathbb{C} \setminus \{ 0 \}$ with nontrivial winding order 
about $z=0$.  In the case of the symmetric spaces $G/K$ with 
$G = \SU(1,1)$ or $G=\SL_2 \R$, and
$K$ the subgroup of diagonal matrices in $G$, the question addressed 
in Theorem \ref{firsttheorem2} 
seems to be more complicated and shall not be discussed here.

\begin{theorem}\label{firsttheorem2}
Up to gauge transformations and admissible coordinate 
changes of the resulting CMC surfaces in $\R^3$, the potential 
$\xi_k$ can be taken so that $k \geq -2$.  Furthermore, amongst 
the cases $k \in Z \cap (-2,\infty)$, any case other than $k=-1$ 
will produce CMC surfaces that close on annular domains  
in $\mathbb{C} \setminus \{ 0 \}$ with nontrivial winding order 
about $z=0$.  
\end{theorem}

\begin{proof}
Choose a solution $L$ to $L^{-1} dL=\xi$ and a gauge $p_+\in
\varLambda_+\SL_2\mathbb{C}$, where $p_+$ is allowed to depend on $z$.  
Changing $L$ to $\tilde L = Lp_+$ gives a solution to $\tilde L^{-1} 
d\tilde L = \tilde \xi$, where 
\begin{equation}\label{gauge}
\tilde \xi =  p_+^{-1}\xi p_+ + p_+^{-1} dp_+ \; . 
\end{equation}
Thus the holomorphic potentials $\xi$ and $\tilde \xi$ make the 
same collection of surfaces via the DPW method.  

Applying the transformation $z \to \frac{1}{z}$, $\xi$ changes to 
\begin{equation}\label{g1}
                        \lambda^{-1}
                        \begin{pmatrix}
                         0&-z^{-2}\\
                         -cz^{-k-2}&0
                        \end{pmatrix}dz \; , 
\end{equation}
and gauging this resulting potential with 
\[ p_{+}=\begin{pmatrix}
         i z^{-1}&0\\
         -i\lambda & -i z
        \end{pmatrix}\in\SL_2 \mathbb{C} \; , 
\]
we get the potential 
\begin{equation}\label{g2}
p_+ ^{-1}\xi p_+ +p_+^{-1}dp_+=\lambda^{-1} \begin{pmatrix}
                                   0&1\\
                                   cz^{-k-4}&0
                                  \end{pmatrix}dz \; . 
\end{equation}
This implies that the cases $k$ and $-k-4$ produce the same surfaces.  
Also, by gauging and coordinate changes, we can see that we 
may assume $c > 0$.  
So without loss of generality, we can restrict to \[ c > 0 \] and 
\[ k \geq -2 \; . \]  

For $k \geq 0$, Smyth surfaces can be produced, which of course are 
well-defined immersions on annular regions with nontrivial winding order 
about $z=0$.  (And in fact, they can extend to $z=0$.)  

The situation for $\xi_{-1}$ is already established in Theorem \ref{2}.  

For $\xi_{-2}$, gauging with 
\[ p_{+,1}=\begin{pmatrix}
         1&0\\
         -\tfrac{\lambda}{2z} & 1
        \end{pmatrix}\in\SL_2 \mathbb{C} \; , 
\]
and then 
\[ p_{+,2} = \begin{pmatrix}
\sqrt{z} & 0 \\ 0 & \frac{1}{\sqrt{z}} 
\end{pmatrix} \in \SL_2 \mathbb{C} \; , 
\]
and then by an appropriate constant diagonal $p_{+,3}$ 
gives the potential 
\[ \sqrt{c} \begin{pmatrix}
0 & \lambda^{-1} \\ 
\lambda^{-1} + \frac{\lambda}{4 c} & 0 
\end{pmatrix} \frac{dz}{z} \; . \]
Further gauging by 
\[ p_{+,4} = \begin{pmatrix}
\Omega^{3/4} & \frac{\lambda}{2 \sqrt{c}} \Omega^{-1/4} \\ 
\frac{\lambda}{2 \sqrt{c}} \Omega^{1/4} & \Omega^{1/4} 
\end{pmatrix}^{-1} \; , \;\;\; \Omega = 1+\frac{\lambda^2}{4c} \; , 
\] and then by \[ p_{+,5} = \exp \left( 
\lambda^{-1} \sqrt{c} (1 - \sqrt{\Omega}) \log z \cdot 
\begin{pmatrix}
0 & 1 \\ 1 & 0 
\end{pmatrix} \right) \] shows that we may assume 
the potential is 
\[ \sqrt{c} \lambda^{-1} \begin{pmatrix}
0 & 1 \\ 1 & 0 
\end{pmatrix} \frac{dz}{z} \; , \]  
and we can choose the CMC surface to be a cylinder 
in $\mathbb{R}^3$, which can of course close to become annular.  

This completes the proof.  
\end{proof}


\begin{thebibliography}{10}

\bibitem{48} D. Brander, W. Rossman and N. Schmitt, 
Holomorphic representation of constant mean curvature surfaces in 
Minkowski space: consequences of non-compactness in loop group methods, 
Adv. Math. 223 (2010), 949--986.  

\bibitem{51} J.F. Dorfmeister, M. Guest and W. Rossman, 
The $tt^*$ structure of the quantum cohomology of $CP^1$ 
from the viewpoint of differential geometry.  Preprint, arXiv:0905.3876.

\bibitem{47} J. Dorfmeister and G. Haak, Investigation and 
application of the dressing action on surfaces of 
constant mean curvature, Quart. J. Math. 51 (2000), 57--73.  

\bibitem{49} J.F. Dorfmeister, F. Pedit and H. Wu, 
Weierstrass type representation of harmonic maps into symmetric spaces,
Comm. Anal. Geom.  6  (1998),  no. 4, 633--668.

\bibitem{45} P. Kellersch,
The Iwasawa decomposition for the untwisted group of loops in semisimple Lie groups
[Eine Verallgemeinerung der Iwasawa Zerlegung in Loop Gruppen, Ph.D. Thesis, TU-Munich, 1999], 
Geometry Balkan Press, 2003,
Electronic version: http://www.mathem.pub.ro/dgds/mono/dgdsmono.htm

\bibitem{**}  S. Kobayashi, 
Real forms of complex surfaces of constant mean curvature,
to appear in Trans. A.M.S.  

\bibitem{34} B. Smyth, 
A generalization of a theorem of Delaunay on constant mean curvature surfaces,
in  Statistical thermodynamics and differential geometry of microstructured materials 
(Minneapolis, MN, 1991), 123--130,
IMA Vol. Math. Appl., 51, Springer, New York, 1993. 

\bibitem{1}  M. Timmreck, D. Ferus and U. Pinkall,
Constant mean curvature planes with inner rotational symmetry in Euclidean $3$-space,
Math. Z.  215  (1994),  no. 4, 561--568.

\end{thebibliography}
\end{document}